\documentclass{article}
\usepackage{latexsym,amsmath,amssymb,amsthm,amsfonts, mathrsfs}
\usepackage[utf8]{inputenc}
\usepackage[T1]{fontenc}
\usepackage{graphicx}

\newtheorem{theorem}{Theorem}
\newenvironment{thmbis}[1]
  {\renewcommand{\thetheorem}{\ref{#1}$'$}%
   \addtocounter{theorem}{-1}%
   \begin{theorem}}
  {\end{theorem}}

\newtheorem{remark}[theorem]{Remark}

\newtheorem{lemma}[theorem]{Lemma}

\newtheorem{question}{Question}
\newtheorem{claim}{Claim}

\begin{document}

\title{\bf Saturated Subgraphs of the Hypercube}
\author{J. Robert Johnson  and Trevor Pinto\thanks{Supported by an EPSRC doctoral studentship.}\\
\small School of Mathematical Sciences,\\[-0.8ex]
\small Queen Mary University of London,\\[-0.8ex] 
\small London E1 4NS, UK.\\
}
\maketitle

\begin{abstract}
We say a graph is \emph{$(Q_n,Q_m)$-saturated} if it is a maximal $Q_m$-free subgraph of the $n$-dimensional hypercube $Q_n$. A graph is said to be \emph{$(Q_n,Q_m)$-semi-saturated} if it is a subgraph of $Q_n$ and adding any edge forms a new copy of $Q_m$.   The minimum number of edges a $(Q_n,Q_m)$-saturated graph (resp. $(Q_n,Q_m)$-semi-saturated graph) can have is denoted by $sat(Q_n,Q_m)$ (resp. $s\text{-}sat(Q_n,Q_m)$). We prove that $ \lim_{n\to\infty}\frac{sat(Q_n,Q_m)}{e(Q_n)}=0$, for fixed $m$, disproving a conjecture of Santolupo that, when $m=2$, this limit is $\frac{1}{4}$. Further, we show by a different method that $sat(Q_n, Q_2)=O(2^n)$, and  that  $s\text{-}sat(Q_n, Q_m)=O(2^n)$, for fixed $m$. We also prove the lower bound $s\text{-}sat(Q_n,Q_m)\geq \frac{m+1}{2}\cdot 2^n$, thus determining $sat(Q_n,Q_2)$ to within a constant factor, and discuss some further questions.
\end{abstract}


\section{Introduction}

Let $F$ be a (simple) graph. We say that a (simple) graph $G$ is \emph{$F$-free} if it contains no subgraphs isomorphic to $F$. If $G$ is a maximal $F$-free subgraph of $H$, we say that $G$ is \emph{$(H, F)$-saturated}. In other words, $G$ is $F$-saturated if it is a subgraph of $H$ and the addition of any edge from $E(H)\setminus E(G)$ forms a copy of $F$. In this context, $H$ is referred to as the \emph{host graph}, $F$ as the \emph{forbidden graph} and $G$ as a \emph{saturated graph}.

The famous Tur\'an problem in extremal combinatorics can be expressed naturally in the language of saturated graphs. The extremal number of $F$, $ex(K_n,F)$, (often written as $ex(n,F)$) is usually defined as the maximum number of edges in an $F$-free subgraph of $K_n$. However, it can equivalently be written as:

\[ex(K_n, F)=\max\{e(G): G\text{ is $(K_n,F)$-saturated}\}.\]

This formulation yields a natural `opposite' of the Tur\'an problem. We define the \emph{saturation number of $F$}, $sat(H, F)$ as: 

\[sat(H,F)= \min\{e(G): G \text{ is $(H, F)$-saturated}\}.\]

A variant of this is the \emph{semi-saturation number}, $s\text{-}sat(H,F)$. We say that a graph is \emph{$(H, F)$-semi-saturated} if $G$ is a subgraph of $H$ and adding any edge from $E(H)\setminus E(G)$ increases the number of copies of $F$. A graph is $(H, F)$-saturated  if and only if it is $(H, F)$-semi-saturated and $F$-free. We define:

\[s\text{-}sat(H,F)=\min\{e(G): G \text{ is $(H, F)$-semi-saturated}\}.\]

The most frequently studied host graph is the complete graph, $K_n$. Since work in the area began with Erd\H{o}s, Hajnal and Moon \cite{erdoshajnalmoon}, many others have studied $s\text{-}sat(K_n,F)$ and $sat(K_n,F)$: see for instance the survey articles of Pikhurko \cite{pikhurko} and of J. Faudree, R. Faudree and Schmitt \cite{faudrees} and the references contained therein. 

In the literature, $sat(K_n,F)$ is often written as $sat(n,F)$  and $(K_n,F)$-saturated is usually written as $F$-saturated. Since the results in this paper concern a different host graph, we will reserve this latter abbreviation for a different meaning.

A much studied variant of the Tur\'an problem was initiated by Erdős in \cite{erdos} and expanded upon by Alon, Krech and Szab\`o \cite{alonkrechszabo}. For a fixed graph $F$, they ask for $ex(Q_n,F)$, the maximum number of edges in an $F$-free subgraph of the $n$-dimensional hypercube, $Q_n$. The most natural case is $F=Q_m$, a fixed cube. This is wide open, even for the case $m=2$. The asymptotic edge density of a maximum $Q_2$-free graph, i.e. $\lim_{n\to \infty} \frac{ex(Q_n,Q_2)}{e(Q_n)}$ was conjectured by Erd\H{o}s \cite{erdos} to be $\frac{1}{2}$. It is still unknown, despite the attention of many authors---see for instance the work of Balogh, Hu, Lidick\'y and Liu \cite{baloghhulidickyliu} and of Brass, Harborth and Nienborg \cite{brassharborthnienborg}.


In this paper, we focus on the saturation and semi-saturation problems, where the host graph is the hypercube and the forbidden graph is a subcube. That is, we study $sat(Q_n,F)$ and $s\text{-}sat(Q_n,F)$. For brevity, we shall often write $F$-saturated (resp. $F$-semi-saturated) rather than $(Q_n,F)$-saturated (resp. $(Q_n,F)$-semi-saturated) in the remainder of this paper, when the value of $n$ is clear or irrelevant.

The best result along these lines is that of Choi and Guan \cite{choiguan}:

\[\limsup_{n\to \infty} \frac{sat(Q_n,Q_2)}{e(Q_n)}\leq\frac{1}{4}.\]

A conjecture that this is best possible, due to Santolupo, was reported in \cite{faudrees}. The same survey article posed the more general question of determining $sat(Q_n,Q_m)$.

The main result of this paper, in Section \ref{zero density}, is the  construction, for all fixed $m$, of $(Q_n,Q_m)$-saturated graphs of arbitrarily low edge density, thus both generalizing and improving the bound of Choi and Guan.

\begin{theorem}\label{generalupper}
For fixed $m$,

\[\lim_{n\to \infty} \frac{sat(Q_n,Q_m)}{e(Q_n)}=0.\]

\end{theorem}

Slightly more precisely, we show $sat(Q_n,Q_m)\leq \frac{c_1}{n^{c_2}} e(Q_n)$, where $c_1$ and $c_2$ are constants depending on $m$. In the case $m=2$, $c_2=6/7$; it is higher for larger values of $m$.

In Section \ref{Q_2section}, we prove a stronger bound for the semisaturation version of  the problem.

\begin{theorem}\label{semisatupper}
For all $n, m$, $s\text{-}sat(Q_n,Q_m)< (m^2+\frac{m}{2}) 2^n$.
\end{theorem}

In the same section, we adapt this proof in the $m=2$ case to remove all copies of $Q_2$ and thus prove a bound on $sat(Q_n,Q_2)$ much stronger than that given by Theorem \ref{generalupper}.

\begin{theorem}\label{Q2upper}
For all $n$, $sat(Q_n,Q_2)< 10\cdot 2^n$.
\end{theorem}

It is easy to see that both these theorems are best possible up to a constant factor, as all $(Q_n, Q_m)$-semi-saturated graphs have minimum degree $m-1$.

In Section \ref{lowerbounds}, we will improve this trivial lower bound, by showing that

\[s\text{-}sat(Q_n,Q_m)\geq \frac{m+1}{2}\; 2^n.\]

In Section \ref{discussion}, we discuss an extension to our zero density upper bound and raise some open questions.

We briefly mention here a somewhat related saturation problem on the cube. Here, $Q_n$ is considered as $\mathcal{P}(X)$, the power set of an $n$ element set, $X$. Let $F$ be a fixed poset. A family $\mathcal{A}\subseteq \mathcal{P}(X)$ is said to be $F$-saturated if there is no subfamily of $\mathcal{A}$ with the same poset structure as $F$, but adding any set to $\mathcal{A}$ destroys this property. Both the maximum and minimum size of such $\mathcal{A}$ have been studied---see for instance Katona and Tarj\'an \cite{katonatarjan} for the former and Morrison, Noel and Scott \cite{morrisonnoelscott} for the latter. 


\section{Preliminaries}

In this section, we introduce terminology, notation and concepts that will be used frequently in the remainder of this paper.

The hypercube $Q_n$ is the graph with vertex set $\{0,1\}^n$, and with edges between each pair of vertices that differ in exactly one coordinate. Alternatively, the vertex set may be considered as $\mathbb{F}_2^n$, the $n$-dimensional vector space over the field with 2 elements. We write $e_1, \dots, e_n$ for the canonical basis of $\mathbb{F}_2^n$ ($e_i$ is the vector with a 1 in the $i^{th}$ coordinate, and 0's elsewhere). We can see that $x$ is adjacent to $y$ if and only if $y=x+e_i$, for some $i\in \{1,\dots, n\}$. 

A subcube of $Q_n$ is an induced subgraph isomorphic to $Q_m$, for some $m\leq n$. A set $S$ of vertices is the vertex set of a subcube if and only if there is some set of coordinates $J\subseteq [n]=\{1,2,3,...,n\}$, and constants $a_j\in \{0,1\}$ for each $j\in J$ such that $(x_1,...,x_n) \in$ $S$ if and only if for all $j\in J$, $x_j=a_j$. \emph{Fixed coordinates} are those coordinates in $J$, whereas \emph{free coordinates} are coordinates that are not fixed. We can thus represent a subcube as an element of $\{0,1,*\}^n$, with stars in the free coordinates, and $a_j$ in the fixed coordinates. As edges can be thought of as $Q_1$'s, we may represent edges as elements of $\{0,1,*\}^n$ in this way. We will say an edge or subcube lies \emph{along} the directions $i_1, \dots, i_k$ if these contain all the free coordinates of the edge or subcube. The \emph{weight} of $x\in V(Q_n)$ is the number of coordinates of $x$ that are 1. 

We may write $Q_{n_1+n_2}$ as $Q_{n_1}\square Q_{n_2}$, the graph Cartesian product of $Q_{n_1}$ and $Q_{n_2}$. In other words, $Q_{n_1+n_2}$ is formed by replacing each vertex of $Q_{n_2}$ with a copy of $Q_{n_1}$. We call these \emph{principle $Q_{n_1}$'s}. Where there was a $Q_{n_2}$ edge $e$, we instead put edges between corresponding vertices of the principle $Q_{n_1}$'s placed at the endpoints of $e$. So we have two types of edges: \emph{internal edges} which have both endpoints in the same principle $Q_{n_1}$ and \emph{external edges} which have endpoints in different principle $Q_{n_1}$'s. Notice that there are $n_1$ directions along which internal edges lie, and $n_2$ directions along which external edges lie. This view of $Q_{n_1+n_2}$ is crucial in the proof of Theorem \ref{generalupper}; we will write  $Q_{n_1+n_2}$ as $Q_{n_1}\square Q_{n_2}$ when we wish to use this viewpoint.

Another way of encapsulating the product nature of $Q_n$ is to write a vertex $v$ as $(v_1| v_2|\dots | v_t)$, where $v_i\in \{0,1\}^{n_i}=V(Q_{n_i})$ and $n_1+\dots+n_t=n$. Two vertices $(v_1| v_2|\dots | v_t)$ and $(u_1| u_2|\dots | u_t)$ are adjacent if and only if there is a $j$ such that $v_j$ and $u_j$ are adjacent as vertices of $Q_{n_j}$ and for all $i\neq j, v_i=u_i$. We will use this notation heavily in Section 4.

An object we shall use in several of our constructions is the \emph{Hamming code}. The properties of Hamming codes that we require are listed below, but see van Lint \cite{vanlint} for  more backgound.
 For our purposes, a Hamming code $C$ can be thought of as a subset of $V(Q_n)$, where $n=2^r-1$ for some $r$, with the following properties:

\begin{enumerate}
\item $C$ is a linear subspace of $\mathbb{F}_2^n$. More precisely, $C$ is the kernel of an $r$ by $n$ matrix $H$ over the field $\mathbb{F}_2$, called a parity check matrix. The columns of $H$ are precisely the non-zero vectors in $\mathbb{F}_2^r$.
\item $|C|=\frac{2^n}{n+1}$.
\item $C$ has minimum distance 3. In other words, $\min\{d(x,y): x,y\in C\}=3$.
\item $C$ is a dominating set for $Q_n$. In other words, every vertex of $Q_n$ is either in $C$ or adjacent to a vertex in $C$. 
\end{enumerate}

Property 1 is usually taken as the definition of a Hamming code; the other properties are simple consequences of it.

A subset $C$ with these properties exists only if $n=2^r-1$ (and when it exists, it is the largest set with Property 3, and the smallest with Property 4). For other values of $n$, we make do with an \emph{approximate Hamming code}. This is any $C\subset V(Q_n)$ satisfying:

\begin{enumerate}
\item $C$ is a linear subspace of $\mathbb{F}_2^n$. More precisely, $C$ is the kernel of an $r=\lceil{\log(n+1)}\rceil$ by $n$ matrix $H$ over the field $\mathbb{F}_2$. $H$ has as columns any  $n$ distinct binary vectors of length $r$.
\item $|C|=\frac{2^n}{2^{\lceil\log_2(n+1)\rceil}}$.
\item $C$ has minimum distance 3. In other words, $\min\{d(x,y): x,y\in C\}=3$.

\end{enumerate}

\section{Zero density bound on $sat(Q_n, Q_m)$} \label{zero density}

In this section, we shall prove a quantitative version of Theorem \ref{generalupper}, of which Theorem \ref{generalupper} is an immediate consequence. 

\begin{thmbis}{generalupper}
For all $m\geq 1$, there exist constants, $c_m$ and $a_m$, such that $sat(Q_n, Q_m)\leq \frac{c_m}{n^{a_m}} e(Q_n)$. More precisely, $a_1=1$ and $a_m=\frac{1}{7\cdot 3^{m-2}}$, for all $m>1$.
\end{thmbis}

Before discussing the proof of Theorem \ref{generalupper}$'$, we sketch a proof of the $\left(\frac{1}{4}+o(1)\right)$ bound of Choi and Guan, as this contains the main ideas of the proof of Theorem \ref{generalupper}$'$. This proof is significantly different from Choi and Guan's, which may be considered more direct. However, our approach, which uses $\frac{1}{3}+o(1)$ density saturated graphs to build $\frac{1}{4}+o(1)$ density saturated graphs,  naturally gives rise to an iterative approach for proving Theorem \ref{generalupper}$'$.

We assume that there exist three $(Q_n,Q_2)$-saturated graphs, $A_1, A_2$ and $A_3$ of $\frac{1}{3}+o(1)$ density, such that every edge of $Q_n$ lies in one of them. We will use these to produce a $\frac{1}{4}+o(1)$ density $(Q_{n+3},Q_2)$-saturated graph $B'$. These $A_i$ are relatively easy to construct---we will require a generalization of them in our proof of Theorem \ref{generalupper}$'$.

We first construct an `almost' $(Q_{n+3},Q_2)$-saturated graph $B$. We  consider $Q_{n+3}$ as $Q_n\square Q_3$. We leave two principle $Q_n$'s corresponding to antipodal vertices of $Q_3$ empty. Around each of these empty $Q_n$, we arrange copies of $A_1, A_2, A_3$, as in the figure below. We also add all external edges with one endpoint in either of the two empty principle $Q_n$'s (as indicated by the bold edges in the figure).

\begin{figure}
\centering
\includegraphics[scale=0.25, trim=0mm 30mm 0mm 20mm, clip]{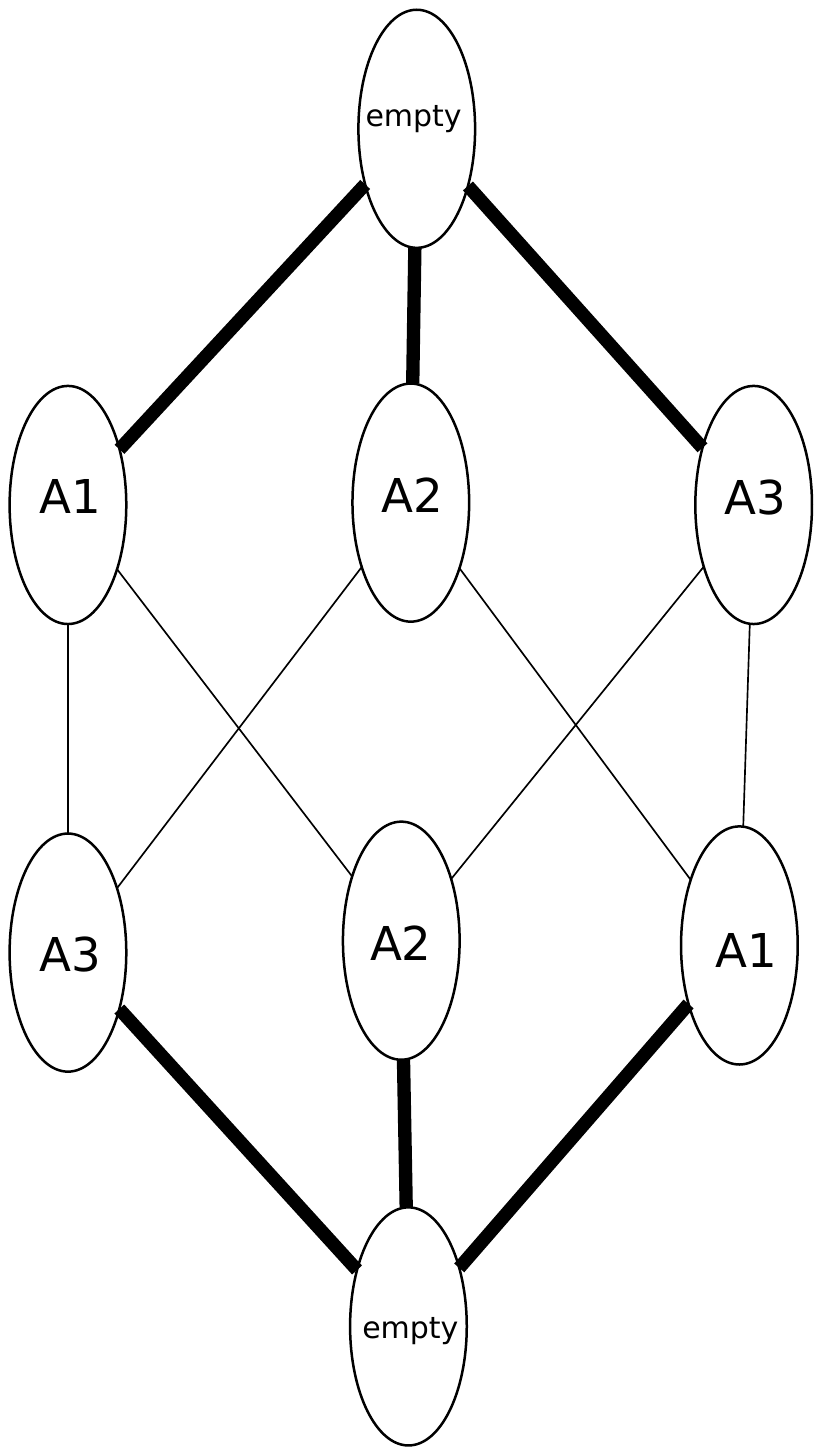}
\caption{The `almost' saturated graph, $B$} 
\end{figure}

The graph constructed has the property that for any edge of an empty $Q_n$, $e$, the corresponding edge, $e'$ is present in one of the $A_i$. So adding $e$ forms a  $Q_2$ comprising $e$, $e'$ and the two external edges that connect corresponding endpoints of $e$ and $e'$. Since the $A_i$ are themselves $Q_2$-saturated graphs, adding any internal edge forms a copy of $Q_2$.

It is easy to see that $B$ is still $Q_2$-free, and a quick calculation shows that $B$ has edge density $\frac{1}{4}+o(1)$. We now prove a simple lemma that allows us to extend $B$ to a $Q_2$-saturated graph. 

\begin{lemma}\label{greedylemma}
Fix $m\geq 2$. Suppose that $G$ is a $Q_m$-free subgraph of $Q_n$ and $S\subseteq E(Q_n)$. Then we can form a $Q_m$-free graph $G'$ by adding no more than $|S|$ edges to $G$ with the property that adding any edge in $S\setminus E(G)$ forms a copy of $Q_m$.
\end{lemma}

\begin{proof}
We order the edges in $S$ arbitrarily. Consider these edges in this order and add them to $G$ if and only if doing so does not form a copy of $Q_m$. Since only edges of $S$ are added by the process, we are done.
\end{proof}

We apply this lemma to $B$, with $S$ being the set of external edges that have not already been added, i.e. those represented by the thin edges in Figure 1. This forms a  $Q_2$-saturated graph, $B'$. Since there are $\frac{3}{n+3} e(Q_{n+3})$ external edges, the asymptotic edge density is still $\frac{1}{4}$.

The proof of Theorem \ref{generalupper}$'$ uses a similar method multiple times to produce $(Q_n,Q_m)$-saturated graphs of arbitrarily low density. In the case where $m=2$,  we assume that we have a collection of $Q_2$-saturated graphs $A_1, \dots, A_k$ of edge density at most $\rho$, such that every edge of $Q_n$ is contained in at least one of the $A_i$. We will view $Q_{n+k}$ as $Q_{n}\square Q_{k}$ and leave several principle $Q_n$ empty. We shall ensure that each empty $Q_n$ is adjacent, for every $i$, to a principle $Q_n$ filled with $A_i$, and add every external edge leaving these empty $Q_n$. This ensures that adding an edge within the empty $Q_n$ forms a copy of $Q_2$. The constraint on the empty principle $Q_n$ is that the set of vertices that we replace with empty $Q_n$'s must have minimum distance 3, and so we employ a Hamming code, enabling us to produce a graph with a lower density, $\rho'$. Of course, to apply this method again, we need several $(Q_{n+k},Q_m)$-saturated graphs of density $\rho'$, which between them cover the edges of $Q_{n+k}$. This turns out to be not much harder, using cosets of the Hamming code.

In the general $m$ case we adapt this method. We would like to use a collection of $A_i$ that cover all the copies of $Q_{m-1}$ in $Q_n$. Such a collection seems hard to construct, but a modification of the argument shows that it suffices to cover almost all copies of $Q_{m-1}$. The other modification is that instead of using empty principle $Q_n$, we fill them with low density $Q_{m-1}$-saturated graphs, which we may assume exist by induction on $m$. We will use the following claim as a key part of the inductive step in proving the theorem.

\begin{claim} \label{increment}
Suppose we have a collection $A_1,\dots, A_k$ of $(Q_n,Q_m)$-saturated graphs, each of density at most $\rho$, and some $n_0$ such that every $Q_{m-1}$ lies along the first $n_0$ directions is within one of these $A_i$. Suppose also that there is a $(Q_n, Q_m-1)$-saturated graph $G$ with no more than $\frac{c_{m-1}}{n^{a_{m-1}}}e(Q_n)$ edges.  Then there is a collection of $k+1$ $(Q_{n+k},Q_m)$-saturated graphs, $B_0, \dots, B_{k}$, such that every $Q_{m-1}$ that lies along the first $n_0$ directions is in one of these $B_i$. Further, each of the $B_i$ has  density at most $(1-\frac{1}{2k})\rho+f(n,n_0)$, where $f$ is a function that tends to zero whenever $n$, $n_0\to \infty$ in such a way that $\frac{n_0}{n}\to 1$.
\end{claim}

A precise upper bound on the densities of the $B_i$ is required for the quantitative part of the theorem; this will be stated at the end of the proof of this claim.

\begin{proof}[Proof of Claim \ref{increment}]

We start by constructing a $k+1$ colouring $c_0$ of $Q_k$, with the colours $0,1,\dots, k$. Fix $C_0$, an approximate  Hamming code in $Q_k$.  We set $c_0(x)=0$ for all $x\in C_0$ and for all $j\in \{1,\dots ,k\}$ and all $x \in C_0$, we set $c_0(x+e_j)=j$. Note that when $k+1$ is not a power of 2 (i.e. when we do not have a genuine Hamming code), this colouring is not fully defined, since $C_0$ is not dominating. For now we assign arbitrary colours other than 0 to these vertices, but we will later decide on these colours.

We write $Q_{n+k}=Q_n\square Q_k$. We  induce from $c_0$ a colouring on the set of principle $Q_n$'s in the natural way. We start forming the graph $B_0$ by placing a copy of $A_j$ in each principle $Q_n$ coloured $j$, for each $j\neq 0$. Also, we add to the graph $B_0$ every external edge with one endpoint in a principle $Q_n$ coloured 0. 

We place a graph isomorphic to  $G$ in each $Q_n$ that is coloured 0 (we will choose which isomorphism later).

Notice that so far, $B_0$ is $Q_m$-free. Indeed, suppose that $B_0$ does contain a $Q_m$. This $Q_m$ cannot lie entirely within a single principle $Q_n$, by our assumption that the $A_i$ are saturated. As we have only added external edges that leave $Q_n$ coloured 0, the $Q_m$ may contain an edge between two principle $Q_n$'s only if one of them is coloured $0$. Since the Hamming code has minimum distance 3, the $Q_m$ must contain edges in exactly two principle $Q_n$'s, one of which is coloured $0$. But such $Q_n$ are $Q_{m-1}$-saturated and thus contain no $Q_{m-1}$, yielding a contradiction.

So far, $B_0$ is not quite $Q_m$-saturated---for instance adding an external edge may not create a copy of $Q_m$. However, we use Lemma \ref{greedylemma} to remedy this. We add at most $\frac{k}{n+k} e(Q_{n+k})$ edges to $B_0$ and we now only need to consider adding internal edges.

Adding an edge within a $Q_n$ coloured $j\neq 0$ forms a $Q_m$, as each $A_j$ is $Q_m$-saturated. Adding an edge within a principle $Q_n$ coloured 0 will form a $Q_{m-1}$ within that $Q_n$. If that $Q_{m-1}$ only uses edges in the first $n_0$ directions, it lies within one of the $A_j$ by the hypothesis of Claim \ref{increment}. Since every principle $Q_n$ coloured zero is adjacent to a principle $Q_n$ of every non-zero colour, a $Q_m$ will be formed. Therefore, we only need to worry about adding edges to $G$ if the $Q_{m-1}$ formed does not lie exclusively along the first $n_0$ directions---we call such edges \emph{bad edges}. We will now show that we may assume there are not very many bad edges.

Apply a random automorphism of $Q_n$ to $G$, our low density $Q_{m-1}$-saturated graph. We call the graph formed $G'\subseteq Q_n$, which is to be placed within a principle $Q_{n}$ coloured 0. Let $e$ be a fixed edge of this principle $Q_n$.

\begin{align*}
\mathbb{P}(\text{$e$ is a bad edge})& \leq 1-\frac{n_0}{n}\cdot\frac{n_0-1}{n-1}\cdot \dots \cdot \frac{n_0-m+2}{n-m+2} \\
&\leq 1-\frac{(n_0-m)^{m-1}}{n^{m-1}} \\
&=\frac{n^{m-1}-(n_0-m)^{m-1}}{n^{m-1}}.
\end{align*}

This tells us that the expected number of bad edges, in each principle $Q_n$ coloured 0, is no more than $\left(\frac{n^{m-1}-(n_0-m)^{m-1}}{n^{m-1}}\right)e(Q_n)$. We now choose the automorphism of $G$ that we left unspecified earlier; we can do this such that we get no more bad edges than the expected number. We use Lemma \ref{greedylemma}, with $S$ being the set of bad edges, to form a graph that we also call $B_0$ that is $Q_m$-saturated.

We now construct the other $B_i$ to cover the required $Q_{m-1}$'s. To construct $B_i$, we repeat the same method used for constructing $B_0$, except we use $C_i:=\{c+e_i:c \in C_0\}$ instead of $C_0$. Note that we can make the arbitrary choices of colours to ensure each principle $Q_n$ is filled with each of the graphs $A_1, \dots, A_k$, in one of the $B_i$. 

It is easy to see that the $B_i$ satisfy the necessary $Q_{m-1}$ condition. Indeed any $Q_{m}\subseteq Q_{n+k}$ along the first $n_0$ directions must lie within a principle $Q_n$. When considered as a subgraph of this $Q_n$, it must lie in a copy of one of the $A_i$---say $A_j$. This principle $Q_n$ is filled with $A_j$ in one of the $B_i$, so we are done. 

It remains only to bound the number of edges in each saturated subgraph, $B_i$. Let $e(A)=\max\{e(A_i)\}, \;e(B)=\max\{e(B_i)\}, \;\rho(A)=\frac{e(A)}{n2^{n-1}}$ and $\rho(B)=\frac{e(B)}{(n+k)2^{n+k-1}}$. In the calculations that follow, we write $a=a_{m-1}$ and $c=c_{m-1}$ for brevity.

Recall that edges were added to each $B_j$ in 4 ways: from copies of $A_i$, from adding external edges, from the $Q_{m-1}$-saturated graphs and from adding bad edges. 

Thus we have:
\begin{align*}
e(B)&\leq  2^k\left( 1-\frac{1}{2^{\lceil \log(k+1)\rceil}}\right) e(A) +\frac{k}{n+k}e(Q_{n+k})\\
&\quad +\frac{2^k}{2^{\lceil \log(k+1)\rceil}} e(Q_n) \left(c_{m-1} n^{-a}+\frac{n^{m-1}-(n_0-m)^{m-1}}{n^{m-1}} \right). \\
\text{Therefore,}\\
\rho(B)&\leq \left(1-\frac{1}{2^{\lceil \log(k+1)\rceil}}\right)\rho(A)+\frac{k}{n+k}\\
&\quad+  \frac{1}{2^{\lceil \log(k+1)\rceil}}\left(c_{m-1} n^{-a}+\frac{n^{m-1}-(n_0-m)^{m-1}}{n^{m-1}} \right)  \\
&\leq \left(1-\frac{1}{2k}\right)\rho(A)+ \frac{k}{n} + \frac{1}{k}\left(c_{m-1} n^{-a}+\frac{n^{m-1}-(n_0-m)^{m-1}}{n^{m-1}} \right).
\end{align*}

Clearly if $n_0$ is large enough, and $n=(1+o(1))n_0$, the last two terms can be arbitrarily small, thus concluding the proof of the claim.

\end{proof}

We now return to prove Theorem \ref{generalupper}$'$.

\begin{proof}[Proof of Theorem \ref{generalupper}$'$]
We use induction on $m$.

\textbf{Base case:} $m=1$. This is trivial---the subgraph of $Q_n$ with no edges is $Q_1$-saturated.

\textbf{Inductive step:} take $m>1$ and assume the Theorem holds for $m-1$-- i.e. there is a $(Q_n, Q_m-1)$-saturated graph $G$ with no more than $\frac{c_{m-1}}{n^{a_{m-1}}}e(Q_n)$ edges.

We first find a collection of subgraphs $A_1,\dots, A_{m+1}$ of $Q_{n_0}$  that satisfy the hypothesis of Claim \ref{increment}, with $\rho=1$. To do this, let $A_i$ initially consist of all edges whose lowest weight endpoint has weight in $\{i,\dots, i+m-2 \} \mod m+1$, and then extend greedily until $A_i$ is $Q_m$ saturated. Each $A_i$ contains every $Q_{m-1}$ whose lowest weight vertex has weight $i \mod m+1$, so every $Q_{m-1}$ is contained in one of these $A_i$. Trivially, we may bound the density of these $A_i$ above by 1, and it is easy to see this is best possible up to a constant.

We now apply Claim \ref{increment} repeatedly, $t$ times. We write $k_i$ and $n_i$ for the value of $k$ and $n$ after the $i^{\text{th}}$ iterate. Clearly, $k_{i+1}=k_i+1, k_0=m+1, n_{i+1}=n_i+k_i$ and $n_t=n_0+\sum_{i=m}^{m+t}i=n_0+O(t^2)$.

After $t$ steps, we end with saturated graphs of density, $\rho$:

\begin{align*}
\rho&\leq \prod_{i=0}^{t-1} \left(1-\frac{1}{2k_i} \right) +\sum_{i=0}^{t-1} \left(\frac {k_i}{n_i} + \frac{c_{m-1}}{k_i} \cdot n_i^{-a}+ \frac{n_i^{m-1}-(n_0-m)^{m-1}}{k_i n_i^{m-1}} \right)\\
&\leq c \prod_{m=1}^{m+t} \left(1-\frac{1}{2i} \right) + \frac{t(m+t+1)}{n_0} + \frac{tc_{m-1}}{m} \cdot n_0^{-a}+ \frac{t}{m} \frac{n_t^{m-1}-(n_0-m)^{m-1}}{n_0^{m-1}} \\
&= c' \cdot \exp\left(-\frac{1}{2} \sum _{i=1}^{t+m} \frac{1}{i}\right)+ O(t^2 n_0^{-1}) +O(t n_0^{-a})+O\left(\frac{t^3}{n_0}\right)\\
&=c''t^{-\frac{1}{2}} + O(t n_0^{-a})+ O(t^3 n_0^{-1}).
\end{align*}

Here, $c, c'$ and $c''$ are constants dependent on $m$. If $m=2$ it is optimal to take $t=n_0^{2/7}$, otherwise $a<\frac{3}{7}$, it is optimal to take $t=n_0^{2a/3}.$

This gives the required bound.

\end{proof}

Note that the better bound for $sat(Q_n,Q_2)$ in the next section can be fed into the induction in the theorem to produce the slightly better bound of $a_m=\frac{1}{7\cdot 3^{m-3}}$.

\section{Bounded average degree constructions} \label{Q_2section}

\subsection{Semi-saturation}

In this section we will prove Theorem \ref{semisatupper}, by constructing for each $m$ a family of $Q_m$-semi-saturated graphs with bounded average degree. Although it seems difficult in general to make these graphs $Q_m$-free, in the $m=2$ case we will use similar ideas to prove Theorem \ref{Q2upper}.

In what follows it will be useful to write $n=m(2^t-1)+r$, where $0\leq r< m2^{t}$, and to let $n_0=2^t-1$. We write a vertex of $Q_n$ as $(v_1|v_2|\dots|v_m| v_{m+1})$, where $v_i\in \{0,1\}^{n_0}$ for $i\leq m$ and $v_{m+1}\in \{0,1\}^r$. The final section of the vector is only included to make the number of coordinates exactly $n$ but otherwise has no importance in the construction.

\begin{proof}[Proof of Theorem \ref{semisatupper}]

Let $C\subseteq \{0,1\}^{n_0}$ be a Hamming Code. We define: 
\[A=\{(v_1|\dots| v_m| v_{m+1}) \in V(Q_n): \exists i\in \{1, m\} \text{ such that } v_i\in C \}.\]
We form $E(G)$ by picking all edges with at least one endpoint in $A$. Note that vertices in $A$ have degree $n$ in $G$; all other vertices have degree $m$. Therefore $e(G)= \frac{1}{2}( (n-m)|A|+m 2^n) \leq \frac{m}{2}( n \frac{2^{n}}{(n_0+1)}+2^n)$. As $\frac{n}{n_0}<2m$,  $e(G)$ satisfies the bounds of the theorem.

We now show that $G$ is $Q_m$-semi-saturated. Assume $e\in E(Q_n)\setminus E(G)$ is  along a direction $i$ in $\{1,n_0\}$ (all other cases can be dealt with similarly). We write the endpoints of the edges as $(v_1|v_2|\dots v_m| v_{m+1})$ and $(v_1'|v_2|\dots|v_m| v_{m+1})$, where $v_1'$ and all of the $v_i$ do not lie in $C$.  Thus for $i=2,3,\dots, m$ there exists $c_i\in C$ adjacent to $v_i$. Consider the $2^m$ points of the form $(x_1|\dots|x_m|v_{m+1})$, where $x_1 \in\{v_1, v_1'\}$ and for $i=2,3,\dots, m$, $x_i\in \{v_i, c_i\}$. These vertices form a subcube of $Q_n$ and all but the endpoints of $e$ are in $A$. Thus when the edge $e$ is added, a copy of $Q_m$ is formed, concluding our proof.

\end{proof}

\begin{remark}
Clearly, when $n=m(2^t-1)$ for some $t$, we get the slightly stronger bound $s\text{-}sat(Q_n,Q_m)\leq \left(\frac{m^2}{2} +\frac{m}{2}\right) 2^n$.
\end{remark}

\subsection{Improved bound for $sat(Q_n,Q_2)$}

In the $m=2$ case, the $Q_2$-semi-saturated graph constructed above consists of all edges incident with vertices in $A=\{(v_1| v_2| v_3) \in V(Q_n): v_1\in C \text{ or } v_2\in C \}$. It is easy to see this contains large subcubes, of the form $(c| *,\dots, *|*,\dots, *)$ or $( *,\dots, *|c|*,\dots, *)$, for $c\in C$. There are other $Q_2$'s in this graph, but those within these large subcubes are hardest to deal with. We prevent subcubes of the first type by only adding edges of the form $\{(c|v), (c|v')\}$, where $c\in \{0,1\}^{n_0}$ and $v\in \{0,1\}^{n-n_0}$ and $v$ has lower weight than $v'$, if $v_1$ has even weight. Of course doing just this alteration means the graph is no longer semi-saturated; we get around this by picking a subset $D$ of $V(Q_{n_0})$ with similar properties to $C$, and adding edges starting at $(d|v_2|v_3)$ if $(v_2|v_3)$ contains an odd number of 1's and if $d\in D$. We make use of the following claim, which allows us to choose a $D$ with the required properties.

\begin{claim}
There exists a $Q_2$-free spanning subgraph, $H$, of $Q_{n_0}$, that has two independent dominating sets, $C, D\subset V(H)=\{0,1\}^{n_0}$, with $C$ disjoint from $D$, where $|C|= 2^{n_0}/(n_0+1)$ and $|D|=3\cdot 2^{n_0}/(n_0+1)$. Further, $H$ only contains edges incident with $C\cup D$ and $e(H)\leq 2^{n_0+1}$.
\end{claim}

We shall prove this claim later, but first we show why it implies the theorem.

\begin{proof}[Proof of Theorem \ref{Q2upper}]

Similarly to before, we write $n=2(2^t-1)+r$, where  $0\leq r < 2^{t+1}$, and let $n_0=2^t-1$. We write an element, $x$, of $\{0,1,*\}^n$ as $(x_1|x_2|x_3)$, where $x_1, x_2\in \{0,1,*\}^{n_0}$ and $x_3\in \{0,1,*\}^r$. We refer to $x_1$ as the first part of $x$, $x_2$ as the second part and so on. We will use this notation particularly when $x$ represents a vertex or an edge of $Q_n$ (it contains no stars or one star).

We start by constructing a graph $G$ that is $Q_2$-free and will then use Lemma \ref{greedylemma} add a `few' edges ($o(2^n)$ edges) to form $G'$, a $Q_2$-saturated graph. As in the proof of Theorem \ref{semisatupper}, we will define a subset, $A$ of the vertices, which will be dominating in $G$:

\[A=\{(v_1|v_2|v_3)\in \{0,1\}^n: v_1\in C\cup D \text{ or } v_2 \in C\cup D\}.\]
The definition of $G$ is slightly more complicated. We add edges to $E(G)$ in three stages, and then delete some of these edges to ensure $G$ is $Q_2$-free.

Firstly, we add all edges $e$ where $e_1\in C$, and the remainder, $(e_2| e_3)$, contains an even number of 1's and a single star, as well as edges where $e_2\in C$ and the remainder, $(e_1|e_3)$ contains an even number of 1's and a single star. We call these Type 1 edges. There are $2 |C|(n-n_0) 2^{n-n_0-2}\leq \frac{(n-n_0)}{2(n_0+1)} 2^n $ Type 1 edges. 

Similarly, we add those edges $e$ where $e_1\in D$ and the remainder, $(e_2|e_3)$ contains an odd number of 1's and a single star, as well as edges where $e_2\in D$ and the remainder contains an odd number of 1's and a single star. We call these Type 2 edges. There are $2  (n-n_0)|D| 2^{n-n_0-2}\leq \frac{3(n-n_0)}{2(n_0+1)} 2^n$ Type 2 edges.

Lastly, we add all edges, $e$ where $e_1$ or $e_2$ is an edge of $H$. There are $2\cdot 2^{n-n_0} e(H)\leq 4\cdot 2^n $ Type 3 edges.

We now delete all edges $e$ which have an endpoint, $(v_1|v_2|v_3)$ such that both $v_1$ and $v_2$ lie in $C\cup D$. Thus $e(G)\leq \left(\frac{2(n-n_0)}{n_0+1}+4\right)2^{n}-\frac{n 2^n}{(n_0+1)^2}.$

Suppose, for contradiction, that $G$ contains a $Q_2$. Note that as all edges of $G$ are incident with a vertex of $A$, this $Q_2$ must contain a vertex $(v_1|v_2|v_3)\in A$, where, without loss of generality, $v_1\in C\cup D$. Note that none of the vertices can have their second part in $C\cup D$, or there is a vertex of the $Q_2$ with both first and second part in $C\cup D$, impossible by our deletion step.

Let $s$ be the number of stars of the $Q_2$ that are in the first part of its vector representation. If $s=2$, all four edges are Type 3 edges, impossible as $H$ is $Q_2$-free.

If instead $s=1$,  suppose the other star is in the second part (the other case is identical). Then we may write the vertices of the $Q_2$ as $(v_1|v_2|v_3)$,  $(v'_1|v_2|v_3)$, $(v'_1|v'_2|v_3)$ and $(v_1|v'_2|v_3)$, where $v_1\in C\cup D$ and $v_2, v'_2\notin C\cup D$. It is easy to see that $v'_1\in C\cup D$. By a parity argument, $v_1$ and $v_1'$ are both in $C$ or both in $D$. But this is impossible as $C$ and $D$ are each $H_0$-independent sets.

Finally, if $s=0$, then we can have only Type 1 edges or only Type 2 edges (depending on whether $v_1\in C$ or $v_1\in D$). But this is impossible by a simple parity argument.

We now show that while $G$ is not quite saturated, it is `almost' saturated. Suppose $e$ is a $Q_n$-edge not incident with $A$. Without loss of generality, the endpoints are $(v_1|v_2|v_3)$ and $(v_1'|v_2|v_3)$, where  $v_1, v_1' v_2, v_3\notin C\cup D$. This is an element of $E(Q_n)\setminus E(G)$. Assume that $(v_1|v_3)$ is even, (the other case is very similar) and that $v_1'$ has higher weight than $v_1$. Then pick $c\in C$ adjacent to $v_2$. $\{(v_1'|v_2|v_3),(v_1'|c|v_3)\}$ and $\{(v_1|v_2|v_3),(v_1|c|v_3)\}$ are Type 3 edges. Also, $\{(v_1|c|v_3), (v_1'|c|v_3)\}$ is a Type 1 edge as $(x|y)$ is even. Thus a $Q_2$ would be formed by adding the edge.

All $Q_n$-edges with exactly one endpoint in $A$ are edges of $G$, so we only need to consider edges where one endpoint, $(v_1|v_2|v_3)$, has $v_1$ and $v_2\in C\cup D$.   There are $\frac{2^n}{n}$ edges of this type, and so we may use Lemma \ref{greedylemma} add them greedily to $G$ to form a $Q_2$-saturated graph $G'$, which has no more edges than the bound in the theorem.

\end{proof}

\begin{remark}
Again, we get a stronger bound for some values of $n$; when $n=2(2^t-1)$ for some $t$, it is easy to see that $sat(Q_n,Q_2)\leq 6\cdot 2^n$.
\end{remark}

We now return to prove the claim.

\begin{proof}[Proof of Claim]

Let $C$ be a Hamming code in $Q_{n_0}$. For $i=1,\dots, n_0$, let $v_i$ be the image of the basis vector $e_i$ under the parity check matrix $M$ of the Hamming code. We may assume that $v_1=(1,0,\dots, 0)$, $v_2=(0,1,0,\dots, 0)$ and $v_3=(1,1,0, \dots, 0)$, as every vector in $\mathbb{F}^t_2$ occurs as a column of $M$ . We shall construct $H$ in four stages, and then prove that it has the required properties.

\begin{enumerate}
\item Add to $E(H)$ every $Q_{n_0}$-edge adjacent to an element of $C$.
\item Add to $E(H)$ every $Q_{n_0}$-edge of the form $\{c+e_1+e_k, c+e_1\}$, where $c\in C$, and where $k\in [4,n_0]$ is such that $v_k$ has a 0 in the first coordinate.
\item Add to $E(H)$ every $Q_{n_0}$-edge of the form $\{c+e_1+e_k, c'+e_2\}$, where $c, c'\in C$, and where $k\in [4,n_0]$ is such that $v_k$ has a 1 in the first coordinate and a 0 in the second coordinate.
\item Add to $E(H)$ every $Q_{n_0}$-edge of the form $\{c+e_1+e_k, c'+e_3\}$, where $c, c'\in C$, and where   $k\in [4,n_0]$ is such that $v_k$ has a 1 in the first coordinate and a 1 in the second coordinate.
\end{enumerate}

Since $C$ is a Hamming code, it is an independent, dominating set and $|{C|=2^{n_0}/(n_0+1)}$. We write $C_i=\{c+e_i: c\in C\}$; in other words, $C_i=M^{-1}(v_i)$. Let $D=C_1\cup C_2\cup C_3$. It is easy to see every edge of $H$ is incident with $C\cup D$.  Since the $C_i$ are disjoint translates of $C$, a Hamming code, $|D|=3\cdot 2^{n_0}/(n_0+1)$. 

Again using that $C_1$ is a translate of a Hamming code, every $x\in V(Q_{n_0})\setminus C_1$ can be written uniquely in the form $c+e_1+e_k$ for $c\in C$ and $k\in [1,n_0]$. The restriction $k\neq 1$ is equivalent to $x\notin C$. The restriction $k\neq 2$ is equivalent to $x\notin C_3$. This is as $M(c+e_1+e_2)=M(c)+M(e_1)+M(e_2)=v_1+v_2=v_3$. Similarly, $k=3$ if and only if $x\in C_2$. Thus steps 2, 3 and 4 ensure $D$ is independent and dominating in $H$.

Notice also that each $x\notin C\cup D$ is $H$-adjacent to exactly 1 element in $D$. Hence $e(H)\leq 2|Q_{n_0}|$, as required. It remains only to show that $H$ is $Q_2$-free. Suppose not. Since we have only added edges with at least one endpoint in $C\cup D$, the $Q_2$ must contain two opposite vertices in $C\cup D$. Since $C$ has minimum distance 3, and since every $x\notin C\cup D$ is adjacent to only 1 element in $D$, one of these vertices is in $D$, and one is in $C$. Thus the vertices of the $Q_2$ may be written in the form $c\in C, c+e_i, c+e_j$ and $c+e_j+e_i\in C_k$, where $i, j\in [4,n_0]$ are such that $v_i+v_j=v_k$, and $k\in \{1,2,3\}$. But it is impossible for all the edges of this $Q_2$ to lie in $e(H)$. Indeed, suppose for example that $k=3$. Then $v_i$ and $v_j$ must both have 1 in the first coordinate and 1 in the second coordinate, impossible if they sum to $v_k$. This concludes the proof of the claim.

\end{proof}

\section{Lower Bounds}\label{lowerbounds}

All the lower bounds in this section are for $s\text{-}sat$; easily $s\text{-}sat(Q_n, Q_m)\leq sat(Q_n, Q_m)$, so the bounds are also valid for $sat$.

If a graph is $(Q_n, Q_m)$-semi-saturated, for $m\geq 2$, it must be connected. Thus it contains a spanning tree for $Q_n$ and so $s\text{-}sat(Q_n, Q_m)\geq 2^n-1$. This shows that Theorems \ref{semisatupper} and \ref{Q2upper} are best possible up to a constant factor.

Another trivial observation improves this for $m\geq 3$: if a graph is $(Q_n, Q_m)$-semi-saturated, it has minimum degree $m-1$. Thus $s\text{-}sat(Q_n, Q_m)\geq \frac{m-1}{2}2^n$.

We do better than both trivial bounds for all $m$.

\begin{theorem}
If $m\geq 2$, $s\text{-}sat(Q_n, Q_m)\geq \left(\frac{m+1}{2}-o(1)\right)2^n $.
\end{theorem}

\begin{proof}
Let $G$ be a $(Q_n, Q_2)$-semi-saturated graph with minimum degree $m-1$; note this contains all $(Q_n, Q_m)$-semi-saturated graphs.
We call a  pair $(v, e)$, where  $v \in V(Q_n), e\in E(Q_n)\setminus E(G))$, \emph{good} if there is a path of length 3 in $G$ linking the endpoints of $e$, that passes through $v$, meaning $v$ is not a start or end vertex of the path.

Note that every non-edge of $G$ is in at least 2 good pairs, whereas each vertex $v$ is in at most $\binom{d(v)}{2}$ good pairs.

Therefore \[\sum_{v\in V(Q_n)} \binom{d(v)}{2}\geq 2 (e(Q_n)-e(G)).\]

Subject to fixed $\sum_v d(v)$, the left hand side is maximized when the degrees are as different as possible. But no degree can be larger than $n$ or smaller than $m-1$. Note that $2e(G)=\sum_v d(v)$, so we have $\frac{2e(G)-2^n}{n-1}$ vertices of degree $n$ in this extreme case.

So certainly 
\begin{align*}
\frac{2e(G)-(m-1)2^n}{n-1} \binom{n}{2} &\geq n2^n-2e(G)\\
(n+2) e(G)-n(m-1)2^{n-1} & \geq n2^n\\
e(G) &\geq \left(\frac{m+1}{2}-o(1)\right) 2^n.
\end{align*}

\end{proof}

\section{Further Questions}\label{discussion}

Having seen that $\lim_{n\to \infty} \frac{sat(Q_n,Q_m)}{n2^{n-1}}=0$, it is natural to ask for a more precise bound---while in Section 4 we have determined $sat(Q_n,Q_m)$ up to a constant, for $m=2$, there is still a wide gap between the best upper and lower bounds for general $m$. In particular, we do not know whether families of $Q_m$-saturated graphs of bounded average degree exist for all $m$. 

\begin{question}
For which $m$ does there exist a constant $c_m$ such that for all $n$, $sat(Q_n,Q_m)\leq c_m 2^n$?
\end{question}

In Section 4, we were able to produce better bounds on $s\text{-}sat(Q_n,Q_2)$ than $sat(Q_n, Q_2)$. Further, the construction we had for $s\text{-}sat$ contained many copies of $Q_2$. This small amount of evidence may suggest that in general, the two are different, even asymptotically.

\begin{question}
Is $sat(Q_n,Q_2)=s\text{-}sat(Q_n,Q_2)$ for all $n$? Does equality hold for all sufficiently large $n$? If not, is $\liminf \frac{sat(Q_n,Q_2)}{2^n}> \limsup \frac{s\text{-}sat(Q_n,Q_2)}{2^n}$?
\end{question}

Recall that all our lower bounds are for $s\text{-}sat$---it seems hard to bound $sat$ more strongly.

Another version of $sat$ that has been studied in the literature (see Section 10 of \cite{faudrees}) (where the host graph is $K_n$) could be studied for this problem. We say that a graph $G\subseteq Q_n$ is \emph{$(Q_n,Q_m)$-weakly-saturated} if we can add the edges in $E(Q_n)\setminus E(G)$ one at a time (in some order) such that every new edge creates at least one new copy of $F$. We write $w$-$sat(Q_n,Q_m)$ for the minimum number of edges a $(Q_n, Q_m)$-weakly saturated graph can have. Clearly, $w$-$sat(Q_n,Q_m)\leq s\text{-}sat(Q_n,Q_m)\leq sat(Q_n,Q_m)$. It is not hard to see, by induction on $n$, that there are many weakly $(Q_n,Q_2)$-saturated trees and so $w\text{-}sat(Q_n,Q_2)=2^n-1$. Indeed, given any $G_1, G_2$, possibly different weakly $(Q_{n-1},Q_2)$-saturated trees, we place them in complementary $Q_{n-1}$'s, and connect any one pair of corresponding vertices. This forms a weakly $(Q_{n},Q_2)$-saturated tree. However, $w\text{-}sat(Q_n,Q_m)$ is in general not known.

\begin{question}
For $m\geq 3$, what is $w\text{-}sat(Q_n,Q_m)$?
\end{question}

In \cite{alonkrechszabo}, Alon, Krech and Szab\`o discuss an interesting  hypergraph type generalization of the Tur\'an problem on the hypercube. We write $Q_n^t$ for the $2^t$-uniform hypergraph with vertex set $\{0,1\}^n$ and edge set consisting of all $t$-dimensional subcubes of $Q_n$.  We say that a subhypergraph $H$ of $Q_n^t$ is $Q^t_m$-free if it contains no subhypergraph isomorphic to $Q_m^t$. As in the usual ($t=1$) case of this Tur\'an problem, they ask how many edges $H$ can have- in particular asking for the limit: $\lim_{n\to \infty} \max\left\{\frac{e(H)}{\binom{n}{t}2^{n-t}}\right\}$. This question is still open, but it is interesting to know that the corresponding saturation problem can be attacked by the same method as the proof of Theorem \ref{generalupper}$'$.

Let $H$ be a subhypergraph of $Q_n^t$. We say that $G$ is $(Q^t_n,Q^t_m)$-saturated if $G$ is $Q^t_m$-free but adding another $2^t$-edge to $G$ forms a subhypergraph isomorphic to $Q^t_m$. In other words, $G$ is a maximal $Q_m^t$-free subgraph of $Q_n^t$. We write $sat(Q_n^t, Q_m^t)$ for the smallest number of edges  a $(Q^t_n,Q^t_m)$-saturated $H$ can have. We can show by the same method as the proof of Theorem \ref{generalupper}$'$ that, for $t\geq 1$ and $s\geq 0$,
\[\lim_{n\to \infty} \frac{sat(Q^t_n,Q^t_{t+s})}{\binom{n}{t}2^{n-t}}=0.\]

As in the proof of Theorem \ref{generalupper}{'} we proceed by induction on $s$ with the
$s=0$ case being trivial. The iteration step analogous to Claim \ref{increment} is
based on the same colouring of principal $Q_n$'s. In each principal
$Q_n$ with colour 0 we place a low density $Q_{t+s-1}^t$-saturated
subgraph of $Q_n^t$. We also add all those $2^t$-edges which contain
$2^{t-1}$ points in some principal $Q_n$ with colour 0. The remainder of
the proof is a straightforward generalisation and the details are left
to the reader.

\subsection*{Acknowledgements}
The second author was supported by an EPSRC doctoral studentship.


\begin{thebibliography}{99}


\bibitem{alonkrechszabo}
N. Alon, A. Krech, T. Szab\`o.
\emph{Tur\'an's theorem in the hypercube.}
SIAM J. Discrete Maths.
\textbf{21}(1) (2007) 66--72.

\bibitem{baloghhulidickyliu}
J. Balogh, P. Hu, B. Lidick\'y and H. Liu.
\emph{Upper bounds on the size of 4- and 6-cycle-free subgraphs of the hypercube.}
European J. Combin.
\textbf{35} (2014)  75--85.

\bibitem{brassharborthnienborg}
P. Brass, H. Harborth and H. Nienborg.
\emph{On the maximum number of edges in a $C_4$-free subgraph of $Q_n$}.
J. Graph Theory.
\textbf{19} (1) (1995) 17--23.

\bibitem{choiguan}
S. Choi and P. Guan.
\emph{Minimum Critical Squarefree Subgraph of a Hypercube.}
Proceedings of the Thirty-Ninth Southeastern International Conference on Combinatorics, Graph Theory and Computing.
\textbf{189} (2008) 57--64.

\bibitem{erdos}
P. Erdős.
\emph{Some problems in graph theory, combinatorial analysis and combinatorial number theory.}
In Graph Theory and Combinatorics. 
B. Bollob\'as, ed.
Academic Press, London.
(1984) 1--17.

\bibitem{erdoshajnalmoon}
P. Erdős, A. Hajnal and J. W. Moon. 
\emph{A Problem in Graph Theory.}
The American Mathematical Monthly.
\textbf{71}(10) (1964) 1107--1110.

\bibitem{faudrees}
J. R. Faudree, R. J. Faudree and R. Schmitt.
\emph{A Survey of Minimum Saturated Graphs.}
The Electronic Journal of Combinatorics.
\textbf{18} (2011).


\bibitem{katonatarjan}
G. O. H. Katona and T. G. Tarj\'an.
\emph{Extremal problems with excluded subgraphs in the $n$-cube.}
Graph Theory, Lagow, Poland, Lecture Notes in Mathematics.
\textbf{1018}, Berlin: Springer
(1983), 84--93.

\bibitem{morrisonnoelscott}
N. Morrison, J. A. Noel and A. Scott.
\emph{On Saturated $k$-Sperner Systems.}
arXiv:1402.5646. 
(2014).

\bibitem{pikhurko}
O. Pikhurko.
\emph{Results and Open Problems on Minimum Saturated Hypergraphs.}
Ars Combinatorica.
\textbf{72} (2004) 435--451.



\bibitem{vanlint}
J. H. van Lint.
\emph{Introduction to Coding Theory.}
Berlin; Springer-Verlag (1999).




\end{thebibliography}
\end{document}